\documentclass[11pt,bezier]{article}
 \usepackage{amsmath,amssymb,amsfonts,euscript,graphicx,amsthm}
                        \usepackage[all]{xy}
                        \usepackage{multicol}
                        \usepackage{enumerate}
                        \usepackage{stackrel}

  \evensidemargin =   -  3 cm \topmargin = 1 cm
  \newtheorem{The}{Theorem}[section]
  \newtheorem{Pro}[The]{Proposition}
  \newtheorem{Lem}[The]{Lemma}
  \newtheorem{Cor}[The]{Corollary}
  \newtheorem{Def}[The]{Definition}
  \newtheorem{Rem}[The]{Remark}
  \newtheorem{Que}[The]{Question}
  
  \newtheorem{Examp}[The]{Example}

    \let\oldproofname=\proofname
   \renewcommand{\proofname}{\textit{\rm\bf\oldproofname}}
  \title{\bf Two Generalizations of   the Wedderburn-Artin Theorem  with Applications  \thanks {The research of the first author was in part supported by
  a grant from IPM (No. 95130413). This research is partially carried out in the IPM-Isfahan Branch.}
  \thanks {2010{ \it Mathematics Subject Classification}. Primary 16D60, 16D70, 16K99 Secondary 16S50, 15B33. }}

  \author{{\bf M. Behboodi$^{{\rm a,b}}$\thanks{Corresponding author.}, {\bf A. Daneshvar}$^{{\rm a}}$ ~   and~} {\bf M. R. Vedadi$^{{\rm a}}$}\\
   {\small{ $^{{\rm a}}$Department of Mathematical Sciences,  Isfahan University of Technology}}\vspace{-1mm}\\
    {\small{ P.O.Box :  84156-83111,   Isfahan,   Iran}}\\
   {\small{ $^{{\rm b}}$School of Mathematics,   Institute for Research in Fundamental Sciences
   (IPM)}}\vspace{-1mm}\\ {\small{ P.O.Box :   19395-5746, Tehran,  Iran}}\vspace{-1mm}\\
   {\small{mbehbood@cc.iut.ac.ir}}\vspace{-1mm}\\
   {\small{a.daneshvar@math.iut.ac.ir}}\vspace{-1mm}\\
   {\small{mrvedadi@cc.iut.ac.ir}}}
  \date{}

\begin{document}
  \maketitle

 \begin{abstract}
  \noindent
  \noindent   We say that an $R$-module  $M$ is  {\it virtually simple} if $M\neq (0)$ and $N\cong M$ for every non-zero submodule $N$ of $M$, and {\it virtually semisimple} if each submodule of $M$  is isomorphic to a direct summand of $M$. We carry out a study of  virtually semisimple modules and modules which are direct sums  of virtually simple modules. Our theory provides two  natural generalizations   of the Wedderburn-Artin Theorem and  an analogous to the classical Krull-Schmidt Theorem. Some applications of these  theorems  are indicated. For instance, it is shown  that the following statements are equivalent for a ring $R$: (i)  Every  finitely generated left (right)  $R$-modules  is virtually semisimple; (ii)  Every  finitely generated left (right)  $R$-modules  is a direct  sum of virtually simple modules; (iii)  $R\cong\prod_{i=1}^{k} M_{n_i}(D_i)$ where $k, n_1,\ldots,n_k\in \Bbb{N}$ and each $D_i$  is a   principal  ideal V-domain; and {\rm (iv)}  Every  non-zero finitely generated left $R$-module  can be written uniquely (up to isomorphism and  order of the factors)  in the form  $ Rm_1 \oplus\ldots\oplus Rm_k$  where  each $Rm_i$ is either a simple $R$-module or   a left  virtually simple direct  summand   of $R$.

      \end{abstract}

 \section{\bf Introduction}

The subject of determining structure of rings and algebras  over which all (finitely generated)
modules  are direct sums of
certain cyclic modules has a long history. One of the first important contributions in this direction is due to Wedderburn \cite{Wed}. He  showed that  every module over a finite-dimensional $K$-algebra $A$ is a direct sum of simple modules  if and only if
 $A\cong \prod _{i=1}^ m M_{n_i}(D_i)$ where $m, n_1, \ldots,n_m\in\Bbb{N}$ and each $D_i$ is finite-dimensional
division algebra over $K$. After that in 1927, E. Artin generalizes the
Wedderburn's theorem for semisimple algebras (\cite{Ar}). Wedderburn-Artin's result is a landmark in the theory of
non-commutative rings. We recall this theorem as follows:

\begin{The}
\textup{(Wedderburn-Artin Theorem).} For a ring R, the following conditions are equivalent:\vspace{1mm}\\
{\rm (1)}   Every left (right)  $R$-module is a direct sum of simple modules.\\
{\rm (2)}   Every finitely generated  left  (right) $R$-module is a direct sum of simple modules.\\
{\rm (3)}   The left (right)  $R$-module $R$  is a direct sum of simple modules.\\
{\rm (4)}  $R \cong \prod _{i=1}^ k M_{n_i}(D_i)$ where $k, n_1, \ldots,n_k\in\Bbb{N}$ and each $D_i$ is a  division ring.
\end{The}

Another one of the important contributions in this direction is due to G. K\"{o}the \cite{Kot}. He considered rings over which all  modules are direct sums of cyclic modules. K\"{o}the in \cite{Kot}  proved the following. We recall that an {\it Artinian} (resp., {\it Noetherian}) {\it ring} is a
ring which is both a left and right Artinian (resp., Noetherian). A {\it principal ideal
ring} is a ring which is both a left and a right principal ideal ring.

\begin{The}
\textup{(K\"{o}the).}  Over an Artinian principal ideal ring, each module is a direct sum of cyclic modules. Furthermore, if a commutative Artinian ring has the property that all its modules are direct sums of cyclic modules, then it is necessarily a principal ideal ring.
\end{The}

Later Cohen and Kaplansky \cite{Coh-Kap} obtained the following result:

\begin{The}
\textup{(Cohen and Kaplansky).} If R is a commutative ring such that each R-module is a direct sum of cyclic modules, then $R$ must be an Artinian principal ideal ring.
\end{The}

However, finding the structure of non-commutative rings each of whose modules is a direct sum of cyclic modules is still an open question;  see \cite[Appendix B, Problem 2.48]{Sab} and  \cite[Question 15.8]{jain0} (for a partial solution, we refer \cite{B-G-M-S}).
Further,  Nakayama in \cite[Page 289]{Nak} gave an example of a non-commutative right Artinian ring $R$ where each right $R$-module is a direct um of cyclic modules but $R$ is not a principal right ideal ring.

Also, the problem of characterizing rings over which all  finitely generated modules are direct sums of cyclic modules (called FGC-rings) was first raised by I. Kaplansky \cite{kap}, \cite{Kap1} for the commutative setting. The complete characterization of commutative FGC rings is a deep result that was achieved in the 1970s. A paper by R. Wiegand and S. M. Wiegand \cite{Weig} and W. Brandal's book \cite{bran} are two sources from which to learn about this characterization.
The corresponding problem in the non-commutative case is still open; see \cite[Appendix B.  Problem 2.45]{Sab} (for a partial solution, we refer \cite{B-B1} and \cite{B-B2}).

 In this paper we say that an $R$-module  $M$ is  {\it virtually simple} if  $M\neq (0)$ and $N\cong M$ for every nonzero submodule $N$ of $M$ (i.e., up to isomorphism, $M$  is the only non-zero submodule of $M$).
 Clearly, we have the following  implications for  $_RM$:
\begin{center}
 $M$ is simple $\Rightarrow$ $M$ is virtually simple $\Rightarrow$  $M$ is cyclic
\end{center}
Note that these implications are irreversible in general when $R$ is not a division ring.

The above considerations motivated us to study rings for which every (finitely generated) module is a direct sum of virtually simple modules. Since  any injective virtually simple module is simple, so  each left  $R$-module is a direct sum of  virtually   simple modules if and only if  $R$ is semisimple (see   Proposition \ref{lem:quasi-injective} and Corollary \ref{cor:every vs}). Now the following  three interesting natural questions arise:

\begin{Que}
 Describe rings $R$ where   each  finitely generated  left  $R$-module is a    direct sum of  virtually simple modules.\vspace{-1mm}
\end{Que}

\begin{Que}
 Describe rings $R$ where  the  left   $R$-module  $R$  is    a direct sum of  virtually simple modules.\vspace{-1mm}
\end{Que}

 \begin{Que}
Whether the Krull-Schmidt Theorem holds for direct sums of  virtually simple modules?
\end{Que}

One goal of this paper is to answer the above questions.

        We note that a semisimple module is a type of module that can be understood easily from its parts. More precisely, a module $M$ is semisimple if and only if every submodule of $M$ is a direct summand.  This property motivates us to study  modules for which every submodule is isomorphic to a direct summand. In fact,  the notions of ``virtually semisimple modules'' and ``completely virtually semisimple modules''    were introduced and studied in our recent work  \cite{B-D-V} as  generalizations of semisimple modules. We recall that an $R$-module $M$ is  {\it virtually semisimple} if each submodule of $M$ is isomorphic to a direct summand of $M$. If each submodule of $M$ is a virtually semisimple module, we call $M$ {\it completely virtually semisimple}.
         We also have the following implications for  $_RM$:
\begin{center}
 $M$ is semisimple $\Rightarrow$ $M$ is completely virtually semisimple $\Rightarrow$ $M$ is virtually semisimple
\end{center}
 These implications are also irreversible in general  (see  \cite[Examples 3.7 and 3.8]{B-D-V}).

If $_RR$  (resp., $R_R$) is a virtually semisimple module, we then say that $R$ is a {\it left} (resp., {\it right}) {\it virtually semisimple ring}. A {\it left} (resp., {\it right})
{\it completely  virtually semisimple ring} is similarly defined (these notions are not left-right symmetric). In  \cite[Theorems 3.4 and 3.13]{B-D-V},   we gave several characterizations of left  (completely)  virtually semisimple rings.

 Clearly,   an $R$-module $M$ is   virtually simple if and only if $M$ is a non-zero indecomposable   virtually semisimple module.  We note that a semisimple module is a direct sum (finite or not) of simple modules,  but it is not true when we replace  ``semisimple"  by ``virtually semisimple" and   ``simple"  by ``virtually simple"  (see Example \ref{exa:important}).
 It is not hard to show that if  every left $R$-module is  virtually semisimple, then  $R$ is a semisimple ring. Nevertheless,  the Wedderburn-Artin theorem motivates us to study rings for which every finitely generated left (right)  module is a virtually semisimple module. In fact, the following  interesting natural questions arise:

 \begin{Que}
 Describe rings  where each  finitely generated  left  $R$-module is   completely  virtually semisimple.\vspace{-1mm}
\end{Que}

 \begin{Que}
 Describe rings  where each  finitely generated  left  $R$-module is    virtually semisimple.\vspace{-1mm}
\end{Que}

\begin{Que}
 Describe rings  where each cyclic  left  $R$-module is    virtually semisimple.
\end{Que}

Therefore, the second goal of this paper is to answer the above  questions,   however,  Question 1.9 remains open to discussion in the non-commutative case.

In Section 2,  we give two generalizations of the Wedderburn-Artin Theorem (Theorems \ref{the:maintheorem} and \ref{the:char fully}). Also, we prove a unique decomposition theorem for finite direct sum of virtually simple modules,  which is an  analogous to the classical Krull-Schmidt Theorem (Theorem \ref{the:Krull-esch}). Section 3
consists of some applications of these  theorems. Our version of the Krull-Schmidt Theorem  applies to prove that every finitely generated complectly virtually  semisimple  module   can be written ``uniquely" as a direct sum  of virtually simple modules  (see Proposition  \ref{pro: com vss is ds}).
Finally, as an important application, we give a  structure theorem for rings whose finitely generated left  (right)  $R$-modules are  direct  sums  of virtually simple modules (Proposition \ref{pro:charvs} and Theorem \ref{the:app1}).

Throughout this paper, all rings are associative with identity and all modules are unitary.  Any unexplained terminology and all the basic results on rings and modules that are used in the sequel can be found in \cite{Ful, Good2, Lam1, Wis}.

  \section{\bf Generalizations of Wedderburn-Artin  and  Krull-Schmidt Theorems}

Let $M$ and $N$  be two $R$-modules. The  notation $N\leq M$ (resp.,  $N\leq_{e} M$)  means that $N$ is a   submodule (resp., an essential submodule). We use the notation $M\hookrightarrow N$ to denote that $M$ embeds in $N$.  An {\it essential monomorphism}, denoted by  $M\overset{ess}\hookrightarrow N$,  from $M$ to $N$ is any monomorphism $f:M\longrightarrow N$ such that $f(M)\leq_e N$.  Also, we use the notation $E(_RM)$ for  the {\it injective hull} of  $M$.

  Following \cite{Good2}, we denote by ${\rm u.dim}(M)$ and ${\rm K.dim}(M)$ the
{\it uniform dimension} and {\it Krull dimension} of a module $M$, respectively.
If $\alpha \geq 0$ is an ordinal number then the module $M$ is said to be $\alpha$-critical provided
${\rm K.dim}(M) = \alpha$ while ${\rm K.dim}(M/N) < \alpha$ for all non-zero submodules $N$ of $M$. A module
is called {\it critical} if it is $\alpha$-critical for some ordinal $\alpha\geq 0$. It is known that critical modules are uniform (see \cite[Lemma 6.2.12]{Mcc}). We say that a
left ideal $P$ of a ring $R$ is {\it quasi-prime} if $P\neq R$ and, for ideals $A, B \subseteq R$, $AB \subseteq P \subseteq A\cap B$
implies that $A \subseteq P$ or $B \subseteq P$.

The following result is very useful  in our investigation.

\begin{Lem}\label{pro:property1} {\rm (See \cite[Proposition 2.7]{B-D-V}).} Let $M$ be a non-zero virtually
semisimple left $R$-module. Then;\vspace{2mm}\\
{\rm (i)}  The following conditions are equivalent.\vspace{1mm}\\
\indent {\rm (1)} ${\rm u.dim}(M)< \infty$.\vspace{1mm}\\
   \indent {\rm (2)} $M$ is finitely generated.\vspace{1mm}\\
\indent{\rm (3)} $M \cong R/P_1 \oplus \ldots \oplus R/P_n$
where $n\in\Bbb{N}$ and each $P_i$ is a quasi-prime left ideal
of $R$ \indent\indent such  that $R/P_i$ is a critical Noetherian
$R$-modules.\vspace{2mm}\\
{\rm (ii)} If $M$ is finitely generated, then  $M\cong  N$ for all $N \leq_{e} M$.\\
  \end{Lem}

\begin{Pro}\label{lem:quasi-injective}
Every quasi-injective virtually semisimple module $M$ is semisimple.
\end{Pro}

\begin{proof}
Assume that $N\leq M$. By  the assumption, $M=K\oplus L$ where $K\cong N$ and $K,L\leq M$. Since   $K$ is a direct summand of $M$, so $K$ is  $M$-injective and so $N$ is $M$-injective. It follows that $N$ is a direct summand of $M$. Thus  $M$ is a semisimple module.
\end{proof}

\begin{Cor}\label{cor:every vs}
The following conditions are equivalent for a ring $R$.\vspace{2mm}\\
{\rm (1)} Every left (right) $R$-module is a direct sum of virtually simple module.\vspace{1mm}\\
{\rm (2)} Every left (right)  $R$-module is virtually semisimple.\vspace{1mm}\\
 {\rm (3)} $R$ is a semisimple ring.
\end{Cor}

\begin{proof}
 $(1) \Rightarrow (3)$. By assumption,   $E(_RR)$ is a direct sum of injective virtually simple $R$-module. Since every injective module is quasi-injective, so by Proposition \ref{lem:quasi-injective}, $E(_RR)$ is a semisimple $R$-module and hence $R$ is semisimple.

\noindent  $(2) \Rightarrow (3)$ can be proven by a similar way.

\noindent  $(3) \Rightarrow (1)$ and $(3) \Rightarrow (2)$ are evident.
\end{proof}

We recall that  the {\it singular submodule}   ${\rm Z}(M)$   of a left (resp., right)  $R$-module $M$ consisting of elements
whose annihilators are essential left (resp., right)  ideals in $R$. An $R$-module $M$ is called a {\it singular} (resp., {\it non-singular}) {\it module} if
${\rm Z}(M)=M$ (resp., ${\rm Z}(M)=0$).

\begin{Lem}\label{pro:zarb}
Let $R_1$ and $R_2$ be rings and $T=R_1\oplus R_2$. Let ``P" denote any one of the properties: finitely generated, singular,
non-singular, projective, injective, semisimple, virtually semisimple and virtually simple. Then by the natural multiplication  every left $T$-module
 $M$ has the form $M_1\oplus M_2$ where each $M_i$ is a left $R_i$-module and the $T$-module $M$  satisfies the property ``P" if and only if each $R_i$-module $M_i$ satisfies the property.
\end{Lem}

\begin{proof}
  For the first part, we just note that $T$ has two central orthogonal idempotent elements $e_1$ and $e_2$ with  $e_1+e_2=1_T$ and $Te_i=R_i$. Thus if $M$ is a left $T$-module then $M=e_1M\oplus e_2M$ where each $e_iM$ is a left  $R_i$-module. Set  $M_i=e_iM$ ($i=1,2$). In this situation any submodule of $M_i$ has the form $e_iK$ for some $K \leq {_TM}$. Thus the proof in the injectivity  case is easily obtained by Baer injective test. For the other cases  there are routine arguments  by using  Soc$(_TM)=$ Soc$(_{R_1}M_1) \oplus $ Soc$(_{R_2}M_2)$,  Z$(_TM)=$ Z$(_{R_1}M_1) \oplus$ Z$(_{R_2}M_2)$ and by the fact that  if $X\oplus Y\cong M$ is an isomorphism of left $T$-modules then $e_iX\oplus e_iY \cong e_iM$   is an isomorphism of left $R_i$-modules.
\end{proof}

\begin{Lem}\label{lem:nonisovs} Let $M$ and $N$ be  virtually simple  $R$-modules. Then;\vspace{2mm}\\
{\rm (i)}  $M$ is a  cyclic  critical   Noetherian  uniform $R$-module.\vspace{1mm}\\
{\rm (ii)}  If $M$ and $N$ are virtually simple $R$-modules with Hom$_R(M,N)=0$ then $M\cong N$ or \indent Z$(N)=N$.\vspace{1mm}\\
{\rm (iii)} If  $M\not\cong N$ and  $N$ is projective, then ${\rm Hom}_R(M,N)=0$.
\end{Lem}
\begin{proof}
(i) Let $M$ be a virtually semisimple module. Clearly $M$ is cyclic and hence $M\cong R/P_i$ for some $i$
  as stated in Lemma \ref{pro:property1}(i). The last statement is now true because critical modules are uniform.

\noindent (ii) Suppose that $_RM$ and $_RN$ are virtually simple and $0\neq f \in$Hom$_R(M,N)$. We have $M/$Ker$f \cong$Im$f\cong N$. Now if Ker$f=0$, then $M\cong N$ and if
Ker$f\neq 0$, then Ker$f \leq_e M$ because $_RM$ is uniform by (i). Hence $N$ must be singular.

\noindent  (iii) By (ii) and the fact that projective modules are not singular.
\end{proof}

Next we have the following lemma.

\begin{Lem}\label{lem:endoplid} Let $M$ be a projective virtually simple $R$-module. Then the
endomorphism ring    ${\rm End}_R(M)$ is a principal left ideal domain.
\end{Lem}

\begin{proof}
This follows from  \cite [Corollary 2.8]{Gor} and  \cite [Theorem 2.9] {B-D-V}.
\end{proof}

 Being a left virtually  semisimple ring is not    Morita invariant (see \cite[Example 3.8]{B-D-V}). Surprisingly,
 being a left  completely virtually  semisimple ring is  a Morita invariant property (see
 \cite[Proposition 3.3]{B-D-V}). In addition,  being (completely) virtually semisimple module is a Morita invariant property (see \cite[Proposition 2.1 (iv)]{B-D-V}). For an $R$-module $M$ and each $n \in \Bbb{N}$, we use the notation  $M^{(n)}$ instead of $M\oplus \cdots \oplus M$ ($n$ times).

We are now in a position to give the following  generalization of the Wedderburn-Arttin Theorem. We remark  that  the equivalences between $(2)$ and $(3)$   below  has been shown in  \cite[Theorem 3.13]{B-D-V}.

\begin{The}\label{the:maintheorem}
 {\bf (First generalization of  the  Wedderburn-Artin Theorem)} The following statements are equivalent for a ring $R$.\vspace{2mm}\\
{\rm (1)}  The  left   $R$-module  $R$   is    a direct sum of  virtually simple modules.\vspace{1mm}\\
{\rm (2)} $R$ is a left   completely virtually semisimple ring.\vspace{1mm}\\
{\rm (3)}  $R \cong \prod _{i=1}^ k M_{n_i}(D_i)$ where $k, n_1, ...,n_k\in \Bbb{N}$ and each $D_i$ is a principal left ideal domain.

 Moreover, in the statement {\rm (3)},  the integers $k,~ n_1, ...,n_k$ and the principal left
ideal domains $D_1, ...,D_k$ are uniquely determined (up to isomorphism) by $R$.
\end{The}

\begin{proof}
$(1) \Rightarrow (3)$.  By assumption,   $R=I_1\oplus \cdots \oplus I_n$ where $n\in\Bbb{N}$ and each $I_i$ is a (projective) virtually simple $R$-module. Grouping these according to their isomorphism types as left $R$-modules, so we
can assume that  $R=I_1^{(n_1)} \oplus \cdots \oplus I_k^{(n_k)}$ where  $1\leq k\leq n$ and $I_i\ncong I_j$
 for any pair $i\neq j$.  Thus, $R\cong {\rm End}_R(R)={\rm End}_R\big(I_1^{(n_1)} \oplus \cdots \oplus I_k^{(n_k)}\big)$.
  Also, by Lemma \ref{lem:nonisovs}(iii), we have ${\rm Hom}_R(I_s,I_t)=0$ for every $s\neq t$. It follows that
$R\cong \bigoplus_{i=1}^k$End$_R\big(I_i^{(n_i)}\big)\cong \bigoplus_{i=1}^k M_{n_i}\big($End$_R(I_i)\big)$. Now
By Lemma \ref{lem:endoplid}, the endomorphism ring $D_j:={\rm End}_R(I_j)$ is a principal left ideal domain for each $1\leq j\leq k$. This shows that
$R \cong \prod _{i=1}^k  M_{n_i}(D_i)$, and the proof is complete.

 \noindent $(3) \Rightarrow (1)$. By Lemma \ref{pro:zarb}, we can assume that $k=1$, i.e., $R=M_n(D)$ where $n\in\Bbb{N}$ and $D$ is a principal ideal domain.   It is known that $R$ is Morita equivalent to $D$ (see for instance  \cite[Page 525]{Lam1}).
  Let $D\overset{\cal F}\approx R$. Then ${\cal F} (D)=D^{(n)}$ is a virtually semisimple $R$-module because $_DD$ is virtually semisimple.
 We set
$$
N_j = \bordermatrix{~ &   &&&j{\rm -th}&& \cr
                   &0&\cdots &0& D&0&\cdots&0 \cr
                  &0&\cdots &0& D&0&\cdots&0 \cr
                  &\vdots&\ddots&\vdots&\vdots&\vdots&\ddots&\vdots& \cr
                   &0&\cdots &0& D&0&\cdots&0 \cr}
                   ~~~~~~~(1\leq j \leq n).
 $$
\indent Then by   the natural matrix multiplication, $N_j$ is a left $R$-module with $N_j\cong D^{(n)}$. Thus each  $N_j$ is virtually simple, i.e.,  $R$ is a direct sum of virtually simple left $R$-modules.

 \noindent $(2) \Leftrightarrow (3)$ and the {\it ``moreover statement"}  are by  \cite[Theorem 3.13]{B-D-V}.
\end{proof}

A ring $R$ is called a {\it left} (resp., {\it right}) V-{\it ring} if each simple left (resp., right) $R$-module is
 injective. We say that $R$ is V-{\it ring} if it is both left and right V-ring.

\begin{Rem}\label{lem:PLID PRID v-dom}
{\rm  Although there
exists an example of a non-domain which is a left $V$-ring but not a right $V$-ring, the question whether
a left (right) V-domain is a right (left) V-domain remains open in general. See \cite[Corollary 3.3]{jain1} where
the authors proved that the answer is positive for principal ideal domains.
}
\end{Rem}

We need the following proposition.

\begin{Pro}\label{cor:SS+pro}
Let  $R=\prod_{i=1}^{k} M_{n_i}(D_i)$ where $k, n_1,\ldots,n_k\in\Bbb{N}$ and
each $D_i$ is a  principal ideal V-domain. Then every finitely generated left  $R$-module is a direct sum of a projective module   and a singular (injective) semisimple module.
\end{Pro}
\begin{proof}
By Lemma  \ref{pro:zarb}, we can assume that $k=1$, i.e., $R=M_{n}(D)$ where $n\in\Bbb{N}$ and $D$ is a  principal ideal V-domain. It is well-known that properties of being projective, being injective, being finitely generated and  being singular are Morita invariant. Since  $D$  is Morita equivalent to $M_{n}(D)$,   so we can assume that $n=1$, i.e., $R=D$. Let   $M$ be  a finitely generated left $D$-module.  By \cite[Theorem 1.4.10]{Cohn1},
$M \cong D/I_1 \oplus \cdots D/I_n \oplus D^{(m)}$ for some non-zero left ideals $I_i$ ($1\leq i\leq n$) of
$D$ and $m\in\Bbb{N}\cup \{0\}$.
Since every non-zero left ideal of $D$ is essential, so each $D/I_i$ is singular.
       It follows that $M=Z(M) \oplus P$ where $P$ is a projective left $D$-module.

      Now let $S$ be any cyclic $D$-submodule in $Z(M)$. Since $D$ is a hereditary  Noetherian ring,
       by \cite[Proposition 5.4.6]{Mcc}, $S$ has a finite length. It follows that ${\rm Soc}\big(Z(M)\big)\leq_e Z(M)$.
      Also since $D$ is Noetherian,  Z$(M)$ and so Soc$\big(Z(M)\big)$ is  finitely generated. Thus the V-domain condition
      on $D$ implies that  ${\rm Soc}\big(Z(M)\big)$ must be a direct summand of Z$(M)$, proving that
       ${\rm Soc}\big(Z(M)\big)=$Z$(M)$.  Therefore,
     $M=Z(M)\oplus P$  where $Z(M)$ is a semisimple (injective)  module and $P$ is a projective module, as desired.
\end{proof}

We are now going to give the following another generalization of the Wedderburn-Artin Theorem.

Let   $R$ be a ring and $M$  an  $R$-module.   We recall  that a submodule  $N$ of $M$ is
({\it essentially}) {\it closed}  if $N\leq_e K\leq M$ always implies $N=K$.
 Also, the module $M$ is called  {\it extending} (or {\it CS-module}) if every closed submodule of $M$ is a direct summand of $M$.
Given $n\in \mathbb{N}$, a uniform $R$-module $U$ is called an $n$-CS$^+$ module if $U^{(n)}$
is extending and each uniform direct summand of $U^{(n)}$ is isomorphic to $_RU$. An integral
domain in which every finitely generated left ideal is principal is called a
{\it left Bezout domain}. {\it Right Bezout} domains are defined similarly, and when both
conditions hold we speak of a Bezout domain.

The following lemmas are needed.

\begin{Lem}\label{lem:cs+}
\textup{(See \cite [Theorem 2.2]{Cla})} Let $R$ be a simple ring. Then $R$ contains a uniform left ideal $U$ such that $_RU$ is
 $2$-CS$^+$ if and only if
 $R$ is isomorphic to the $k\times k$ matrix ring over a Bezout domain $D$ for some $k \in \mathbb{N}$.
\end{Lem}

\begin{Lem}\label{lem:Bezout Ore}
\textup{(\cite[Proposition 2.3.17]{Cohn1})}
If $R$ is a right Bezout domain then $R$ is right Ore domain.
\end{Lem}

\begin{Lem}\label{lem:emmbed}
\textup {(See \cite[Lemma 1]{Cam-Coz})} Let $R$ be a semihereditary and  Goldie ring with
 classical quotient ring $Q$. Let $S$ be a simple right $R$-module.
Then
 $S$ is finitely presented if and only if
$S$ may be embedded in the module $(Q/R)\oplus R$.
\end{Lem}

\begin{Lem}\label{lem:emmbednoeth}
\textup {(See \cite[Theorem 2]{Cam-Coz})} Let $R$ be a left Noetherian, left hereditary, semiprime
     right Goldie with classical quotient ring $Q$. Then  $R$ is right Noetherian if and only every simple right $R$-module
     can be embedded in $(Q/R)\oplus R$.
\end{Lem}

 \begin{The}\label{the:char fully}
 \textup{({\bf Second  generalization of  the  Wedderburn-Artin Theorem).}}
 The following statements are equivalent for a ring $R$.\vspace{2mm}\\
  {\rm (1)} All finitely generated left  $R$-modules are virtually semisimple.\vspace{1mm}\\
{\rm ($1^\prime$)} All finitely generated right  $R$-modules are virtually semisimple.\vspace{1mm}\\
   {\rm (2)} All finitely generated left $R$-modules are completely virtually semisimple.\vspace{1mm}\\
   {\rm ($2^\prime$)} All finitely generated right  $R$-modules are completely virtually semisimple.\vspace{1mm}\\
    {\rm (3)} $R \cong \prod_{i=1}^{k} M_{n_i}(D_i)$ where each $D_i$ is a principal ideal V-domain.
\end{The}

\begin{proof} Since the statement (3)  is symmetric, we only need to prove $(1) \Leftrightarrow (2)  \Leftrightarrow (3)$.

\noindent $(1) \Rightarrow (2)$ is by Lemma \ref{pro:property1} (not that every finitely generated virtually semisimple module is  Noetherian).

\noindent $(2) \Rightarrow (1)$ is evident.

\noindent $(2) \Rightarrow (3)$.  By assumption, $R$  is left completely virtually semisimple and so by Theorem \ref{the:maintheorem}, $R \cong \prod_{i=1}^{k} M_{n_i}(D_i)$ where $k\in\Bbb{N}$ and
each $D_i$ is a principal left ideal domain. By Remark \ref{lem:PLID PRID v-dom}, it suffices to prove that
each $D_i$ is a principal right ideal domain and a left V-domain. Let $D=D_i$ for some $i$.
 By (2) and the fact that ``completely virtually semisimplity is a Morita invariant property for  modules", we deduce that all finitely generated
 left $D$-modules are also completely virtually semisimple.

 Now let $S$ be a simple $D$-module. Assume that $E=E(S)$ is the injective hull of $_DS$ and $C$ is a cyclic $D$-submodule of $E$. Since $S\leq_e E$, we have $S\leq_e C$ and by Lemma \ref{pro:property1}(ii), $C \cong S$. It follows that $E=S$ and hence $D$ is a left $V$-domain.

  We claim that the left $D$-module $D\oplus D$ is extending. To see this, assume that $N$ is a closed $D$-submodule of $D\oplus D$. If ${\rm u.dim}(N)=2$ then $N$ is an  essential submodule of $D\oplus D$ and hence  $N=D\oplus D$. If ${\rm u.dim}(N)=1$,  then  by \cite[Theorem 6.37]{Lam1},   ${\rm u.dim}((D\oplus D)/N)=1$. Set  $U=(D\oplus D)/N$. Then by assumption, $U$ is a finitely generated uniform left virtually semisimple $D$-module and so  $U\cong D/P$ where $P$ is a  left ideal of $D$ by Lemma \ref{pro:property1}(i). Now if $Z(U)=K/N$ where $N\leq K\leq D\oplus D$,    then  by  Kaplansky's Theorem  \cite[Theorem 2.24]{Lam1}, $K$ is  a free (projective) left $D$-module.
   The singularity of $K/N$ implies that $N \leq_e K$. Since $N$ is closed, we have $N=K$ and hence $U$
   is a non-singular left $D$-module. It follows that $P=0$ (because $D$ is a principal left ideal domain and every non-zero left ideal in $D$ is essential). Thus $U\cong D$ and so $_DU$ is projective. This shows that $N$ is a direct summand
   of $D\oplus D$. Therefore, $D\oplus D$ is a left extending $D$-module, as desired.

   It is now clear that $D$ is a left  $2$-CS$^+$ $D$-module. Also since $D$ is a left V-domain, it is a simple ring.
Thus    by Lemma \ref{lem:cs+}, $D$ is a right Bezout domain. To complete the proof it now remains to
 prove that $D$ is a right Noetherian ring.

 By Lemma \ref{lem:Bezout Ore}, $D$ is a right Ore domain. Since $D$ is a left hereditary ring,  by  \cite[Corollary 12.18]{Dun}  $D$ is a right semihereditary ring.   We are applying Lemma \ref{lem:emmbednoeth} to
  show that $D$ is right Noetherian.
  Now assume that $S$ is a right simple $D$-module.  If $S\cong D$ then $D$ is semisimple and we are done.
   Thus we can assume   $S\cong D/P$ where $P$ is a non-zero maximal right ideal $D$. Since $D$ is right Ore, so
   $P$ is a right essential ideal of $D$ and hence $S_D$ is torsion. Thus by \cite[Proposition 5.3.6]{Cohn1},
   $S_D$ is finitely presented. It follows that  $S$
   can be embedded in the right $D$-module $(Q/D)\oplus D$ by Lemma \ref{lem:emmbed} where $Q$ is the
    classical quotient ring of $D$, and the proof is complete.

\noindent  $(3) \Rightarrow (1)$. By Lemma \ref{pro:zarb}, we can assume that $k=1$, i.e., $R=M_{n}(D)$ where $n\in\Bbb{N}$ and $D$ is a  principal ideal V-domain.  Since   being completely virtually semisimple is  Morita invariant property,  it is enough to show that every finitely generated left $D$-module is completely virtually semisimple. Assume that $M$ is a finitely generated left $D$-module. By Proposition \ref{cor:SS+pro}, $M=S\oplus P$ where $S$ is a semisimple $D$-module and $P$ is a projective $D$-module. Thus by
 \cite[Propositions 3.3 (i)]{B-D-V}, $P$ is virtually semisimple.
 We can assume that  ${\rm Soc(D)}=0$ (otherwise $D$ is a division ring and we are done), so we can deduce that ${\rm Soc}(P)=0$ and by
 \cite[Propositions 2.3 (ii)]{B-D-V}, $M$ is virtually semisimple and the proof is complete.
         \end{proof}

The following example, originally from Cozzens \cite{Coz},  shows that there are principal ideal $V$-domains which are not division rings.

\begin{Examp}\label{exa:cozzen}
\textup{(\cite[Example of Page 46]{jain}) Let $K$ be a universal differential field with derivation
$d$ and let $D = K[y;d]$ denote the ring of differential polynomials in the indeterminate $y$ with coefficients in $K$, i.e., the additive group of $K[y;d]$ is the additive
group of the ring of polynomials in the indeterminate $y$ with coefficients in field
$K$, and multiplication in $D$ is defined by: $ya = ay + d(a)$ for all $a$ in  $K$. It is shown that $D$ is both left and
right principal ideal domain,  the simple left $D$-modules are precisely of the form
$V_a =D/D(y-a)$ where $a$ in $K$ and each simple left $D$-module is injective . Hence $D$ is a left $V$-ring. Similarly, $D$ is a right $V$-ring.}
\end{Examp}

%Next we need the  following lemma.
 In the following, we obtain a uniqueness decomposition theorem  for finite direct sum of virtually simple modules,  which is analogous to the classical Krull-Schmidt Theorem for direct sum decompositions of modules.

\begin{Lem}\label{pro:embbed}
Let  $M=V_1\oplus \cdots \oplus V_n$  be a direct sum of  virtually simple left $R$-modules. Then:\vspace{2mm}\\
{\rm (i)} If $N\leq_e {_RM}$ then $M\hookrightarrow N$.\vspace{1mm}\\
{\rm (ii)} If    $0\neq N \leq {_RM}$ then there is an index $j$ such that $V_j \hookrightarrow N$.
\end{Lem}
\begin{proof}
(i) Assume that $N\leq_e M$. Then   $N \cap V_i\neq 0$ and so  $N \cap V_i \cong V_i$ for each $i$. This shows that
$M\cong \bigoplus_{i=1}^n (N\cap V_i) \subseteq N$.

\noindent (ii) Assume that   $0\neq N \leq M$. We can prove that the result by induction on $n$. Just consider the cases
$N \cap V_i\neq 0$ or $N \cap V_1\neq 0$.
\end{proof}

Let $M$ and $N$ be $R$-modules. We say that $M$ and $N$ are {\it $R$-subisomorphic} if $M\hookrightarrow N$ and $N\hookrightarrow M$.

\begin{The}\label{the:Krull-esch} {\rm {\bf (The Krull-Schmidt Theorem for virtually  simple  modules)}}.
Let   $M=V_1\oplus \cdots \oplus V_n$ and $N=U_1\oplus \cdots \oplus U_m$ where all  $V_i$'s and  $U_j$'s are  virtually simple modules.
 If $M$ and $N$ are $R$-subisomorphic, then   $n=m$ and  there is a  permutation $\sigma$ on $\{1,...,n\}$ such that $U_i \cong V_{\sigma(i)}$.
 \end{The}

\begin{proof}
Since virtually simple modules are uniform (Lemma \ref{lem:nonisovs}(i)), so
by our assumption, $m=$ u.dim$(N) \leq$ u.dim$(M)=n$ and vice versa. Thus
  $m=n$. Without loss of generality, we can assume that $M=X_1\oplus \cdots \oplus X_l$ and $N=Y_1\oplus \cdots \oplus Y_t$
where $X_i = V_i^{(m_i)}$ ($1\leq i\leq l$) with $V_i \ncong V_s$
($ i\neq s$)
 and $Y_j= U_j^{(n_j)}$ ($1\leq j \leq t$) with $U_j \ncong U_k$
 ($j\neq k $).
 Note that if $V_i \hookrightarrow U_k$ and $X_i \cap \big(\Sigma_{j\neq k} Y_j\big)=W$, then $W=0$. Otherwise,   there is a non–zero embedding $W\hookrightarrow X_i$ and so by Lemma \ref{pro:embbed}(ii), $V_i \hookrightarrow W$.
  It follows that $V_i \hookrightarrow \Sigma_{j\neq k} Y_j$ and hence $V_i\hookrightarrow U_h$ for some
 $h\neq k$, a contradiction. Thus we can conclude that for each $i\in \{1,\cdots, l\}$ there exists a unique $k\in \{1,\cdots, t\}$ such that $X_i \hookrightarrow Y_k$. Similarly, each $Y_j$ can be embedded in only one $X_i$'s.  This shows that
 $l\leq t$, and $t\leq l$, i.e., $t=l$. Clearly $V_i \hookrightarrow U_k$ if and only if $V_i\cong U_k$. Thus
 it is enough to show that $m_i=n_k$ when $X_i \hookrightarrow Y_k$.
  Again consider that if $X_i \hookrightarrow Y_k$ and $Y_k \hookrightarrow X_s$, then Lemma \ref{pro:embbed} proves
  that $i=s$. This shows that u.dim$(X_i)=m_i \leq$ u.dim$(Y_k)=n_k$ and vice versa, and hence the proof is now complete.
  \end{proof}

 \section{\bf Some applications}

We  give a  structure theorem for rings  over which every finitely generated  module is a direct sum
 of virtually simple modules.  Such rings form a proper subclass of the class of FGC rings. As an application of
 Theorem \ref{the:Krull-esch}, we first show that every completely virtually simple module is uniquely  (up to isomorphism) a direct sum of virtually simple modules, but the converse is not true in general.

\begin{Pro}\label{pro: com vss is ds}
 Every finitely generated completely virtually  semisimple  module is a direct sum of virtually simple modules. Up to a permutation, the virtually simple components in such a direct sum are uniquely
  determined up to isomorphism.
\end{Pro}

\begin{proof}
Assume that  $M$ is a finitely generated complectly virtually  semisimple  module. By Theorem \ref{the:Krull-esch}, it suffices to show that $M$ is a finite direct sum of virtually simple modules.   If $M$ is virtually simple then we are done. Assume that $M$ is not virtually simple. Thus there is a non-zero submodule $N$ of $M$ such that
$M\ncong N$. By assumption, $M=U\oplus W$ where $U\cong N$ and $0\neq W \leq M$. If $U$ and $W$ are virtually simple then we are done.
If not, without lose of generality, assume that $U$ is not virtually simple. So $U$ has a non-zero submodule $N_1\ncong U$. By assumption, $U$ is again virtually semisimple and  so $U=U_1\oplus W_1$ such that $N_1\cong U_1$ and $0\neq W_1 \leq U$.
It follows that $M=U_1\oplus W_1 \oplus W$.
  If one of the
$U_1, W_1$ or $W$ is not virtually simple,  for example $U_1$,  then we may  repeat the above argument with respect to $U_1$ and continue inductively. Since $M$ is a finitely generated virtually semisimple module, ${\rm u.dim}(M)< \infty$ by Lemma \ref{pro:property1}(i) and hence, we
  obtain virtually simple submodules $K_1,\ldots, K_n$ such that $M=\bigoplus _{i=1 }^n K_i$,  and the proof is completed.
\end{proof}

Let $R$ be a ring and $M$ be a left $R$-module. If $X$ is an element or a subset of $M$, we define the {\it annihilator} of $X$ in $R$ by ${\rm
Ann}_R(X) = \{r \in R~|~rX = (0)\}$. In the case $R$ is non-commutative and $X$ is an element or a subset of an $R$, we
define the {\it left annihilator} of $X$ in $R$ by ${\rm l.ann}_R(X) = \{r \in R~|~rX = (0)\}$ and the {\it right
annihilator} of $X$ in $R$ by ${\rm r.ann}_R(X) = \{r \in R~|~Xr = (0)\}$.

The following example shows that the converse of Proposition \ref{pro: com vss is ds} does not hold in general.

\begin{Examp}\label{exa:important}
 {\rm
Let $F$ be a field and we set $R=F[[x,y]]/\langle xy\rangle$. It
is clear that $\mathcal{M}=RX \oplus RY$ is a maximal ideal of
$R$ where $X= x+ \langle xy\rangle$ and $Y=y+\langle xy\rangle$.
It is easily see that ${\rm Spec}(R)=\{RX, RY, \mathcal{M} \}$.
Consider $\mathcal{M}$ as an $R$-module and hence $\mathcal{M}
\cong
{R /{\rm Ann_R}(X)} \bigoplus {R /{\rm Ann_R}(Y)}=R/RY \bigoplus
R/RX$. By \cite[Theorem 2.1]{Gil}, the rings $R/RX$ and $R/RY$
are principal ideal domains because they have principal prime
ideals. We show that $_R\mathcal{M}$ is not virtually
semisimple. Note that $X^2 =X(X+Y)$ and $Y^2=Y(X+Y)$ and so we
have
$RX^2 \oplus RY^2 \leq R(X+Y) \leq {_R\mathcal{M}}$. It follows
that ${\rm u.dim}(_R\mathcal{M})={\rm u.dim}(R(X+Y))$ or
equivalently $R(X+Y) \leq_e {_R\mathcal{M}}$. Now if
$_R\mathcal{M}$ is virtually semisimple we must have $R \cong
R(X+Y)\cong \mathcal{M} $, but $\mathcal{M}$ is not cyclic. Therefore $_R\cal M$ is not virtually semisimple.}
\end{Examp}

The following result provides a plain structure for virtually simple modules  over $M_{n}(D)$ where $n\in \Bbb{N}$ and  $D$  is a   principal  ideal V-domain.

\begin{Cor}\label{pro:charvs}
 Let  $R\cong M_{n}(D)$ where $n\in \Bbb{N}$ and  $D$  is a   principal  ideal V-domain. Then a
 left $R$-module $M$ is virtually simple if and only if
$$
M\cong \begin{pmatrix}
     D/P  \\
    D/P\\
    \vdots \\
    D/P
  \end{pmatrix}
$$ where $P$ is a  maximal left ideal of $D$ or $P=(0)$.
\end{Cor}

\begin{proof}
This is obtained by Proposition \ref{cor:SS+pro} and the familiar correspondence between modules over $D$ and
 $M_n(D)$.
   \end{proof}

%Now we are in a position to prove the following structure theorem which is  an important application of the second    generalization of the Wedderburn-Arttin Theorem (see Theorem \ref{the:char fully}) and our version of the Krull-Schmidt Theorem (see Theorem \ref{the:Krull-esch}).
\vspace{0.7mm}

\begin{The}\label{the:app1}
The following statements are equivalent for a ring $R$.\vspace{2mm}\\
{\rm (1)}  Every  finitely generated left  $R$-modules  is a direct  sum of virtually simple modules.\vspace{1mm}\\
{\rm ($1^\prime$)} Every  finitely generated right  $R$-modules  is a direct  sum of virtually simple modules.\vspace{1mm}\\
{\rm (2)}  $R\cong\prod_{i=1}^{k} M_{n_i}(D_i)$ where $k, n_1, ...,n_k\in \Bbb{N}$ and each $D_i$  is a
 principal  ideal V-domain.\vspace{1mm}\\
 {\rm (3)} Every finitely generated left $R$-modules is uniquely (up to isomorphism)  a   direct   sum  \indent of cyclic left  $R$-modules that are either simple or  virtually simple direct summand of
 \indent $_RR$.\vspace{1mm}\\
{\rm ($3^\prime$)}Every finitely generated right  $R$-modules is uniquely (up to isomorphism)  a   direct   sum  \indent of cyclic left  $R$-modules that are either simple or  virtually simple direct summand of
 \indent $_RR$.\vspace{1mm}\\
{\rm (4)} Every finitely generated left $R$-module is an extending module that embeds in a direct  \indent sum  of virtually simple modules.
\vspace{1mm}\\
{\rm ($4^\prime$)} Every finitely generated right $R$-module is an extending module that embeds in a direct \indent  sum  of virtually simple modules.
\end{The}

\begin{proof}
Since the statement (2)  is symmetric,  we only need to prove $(1)\Leftrightarrow (2)\Leftrightarrow (3)\Leftrightarrow (4)$.

\noindent $(1)\Rightarrow (2)$. By Theorem \ref{the:char fully}, it suffices to prove that all finitely generated $R$-modules are
virtually semisimple. By Theorem \ref{the:maintheorem}, $R$ is left Noetherian. Let $M$ be a
 finitely generated left $R$-module   and $K\leq M$. It is well-known that $K\oplus L \leq_e M$ for some $L\leq M$.  By our assumption the modules  $M$ and  $N:=K\oplus L$ are direct sum of virtually simple modules. Thus  by  Lemma \ref{pro:embbed}, $M$ and $N$ are subisomorphic  and so by Theorem \ref{the:Krull-esch}, $M\cong N$; proving that $_RM$ is virtually semisimple.

\noindent $(2)\Rightarrow (3)$. By Theorem \ref{the:char fully} and Proposition \ref{pro:charvs}, every finitely generated left $R$-module
 is a direct sum of a semisimple  module and a completely virtually semisimple projective module.
  Thus (3) is obtained by
Proposition \ref{pro: com vss is ds} and Theorem \ref{the:Krull-esch}.

\noindent $(3)\Rightarrow (4)$. Since every simple $R$-module is either singular or projective, the condition (3) shows that every
finitely generated left  $R$-module is a direct sum of a singular and a projective  module.
 Thus (4) is obtained by [14, Corollary 11.4].

\noindent  $(4)\Rightarrow (1)$.  Let $M$ be a finitely generated left $R$-module. By assumption, $M$ is extending with finite uniform dimension. Thus by \cite[Lemma 6.43]{Lam1}, $M$ is a direct sum of uniform modules. So it is enough to show that every finitely generated uniform left $R$-module is virtually simple. Note that if $U$ is a finitely generated  uniform left $R$-module and    $U\hookrightarrow\oplus_{i=1}^k V_i$ where each $V_i$ is a virtually simple $R$-module,     then by induction we can show that  that  $U\hookrightarrow V_j$ for some $j$. It follows that  $U\cong V_j$ and the proof is complete.
\end{proof}

Let $R$ be a ring and $M$ an $R$-module.  An $R$-module $N$ is {\it  generated}  by $M$ or {\it $M$-generated}  if there exists
an epimorphism $M^{(\Lambda)}\longrightarrow N$ for some  index set $\Lambda$.
An $R$-module $N$ is said to be {\it subgenerated} by $M$  if $N$ is isomorphic to a submodule
of an $M$-generated module. For an $R$-module $M$, we denote by $\sigma[M]$ the {\it full subcategory} of $R$-Mod whose objects are all $R$-modules subgenerated by $M$. It is clear that if $M = R$ then $\sigma[M]$ coincides with the category $R$-Mod.

\begin{Rem} {\rm  As another application of the theory of  virtually semisimple modules
  we  shows that  the term ``cyclic" must be removed from  statement (f) of \cite[Proposition 13.3]{Dun}. In fact,
   in Example \ref{exa:counterexa exten} we show that the  following       statements are not equivalent.}\vspace{2mm}\\
{\rm (1)} {\it Every module $N \in \sigma[M]$ is an extending module.}\vspace{1mm}\\
{\rm (2)} {\it Every cyclic module in $\sigma[M]$ is a direct sum of an $M$-projective module and a
semisim-\ \indent ple module.}
\end{Rem}

\begin{Examp} \label{exa:counterexa exten}
{\rm
 Assume that ring $D$ is the same as in Example \ref{exa:cozzen} (example attributed to Cozzens).  It is clear that $\sigma[D]=D$-Mod and since $D$ is not left Artinian, so by \cite[Proposition 13.5, Part g]{Dun},  there is a left $D$-module $M$ such that $M$ is not an extending module. On the other hand,  $D$ is a principal ideal $V$-domain and so  by Corollary \ref{cor:SS+pro}, every cyclic left  $D$-module  is a direct sum of a projective module  and a semisimple module, and hence in the above (2) does not imply (1).}\end{Examp}

We note that  the class of virtually simple modules
is not closed under homomorphic image. For example the $\Bbb Z$-module $\Bbb Z/4\Bbb Z$ is not virtually semisimple but
$\Bbb Z$ is clearly completely virtually semisimple $\Bbb Z$-module.  Thus give the following definitions.

  \begin{Def}
 {\rm An  $R$-module $M$ is called}  fully  virtually semisimple {\rm if for reach $N\leq M$, the $R$-module $M/N$ is virtually semisimple. If $_RR$ (resp., $R_R$) is fully  virtually semisimple, we then  say that  $R$ is a {\it left} (resp., {\it right})  {\it fully  virtually semisimple}. Also, a ring  $R$ is called  a {\it fully  virtually semisimple ring} if it is both a left and right fully  virtually semisimple ring.}
\end{Def}

By Proposition \ref{pro: com vss is ds} and the next proposition,
we have the following irreversible implications for an $R$-module $M$:
\begin{center}
 $M$ is fully virtually semisimple $\Rightarrow$ $M$ is completely virtually semisimple  $\Rightarrow$ $M$ is a finite direct sum of virtually simple modules
\end{center}

 \begin{Pro}\label{pro:left fully completely}
Every  finitely generated fully virtually semisimple module $M$ is  completely virtually semisimple.
 \end{Pro}

 \begin{proof}
 Assume that $K\leq M$. It is well-known that there exists $L\leq M$ such that $L\oplus K \leq_{e} M$
 and $K\cong (L\oplus K) / L \leq_{e} M/L$. Since $M/L$ is finitely generated virtually semisimple, so $K\cong M/L$, by Lemma \ref{pro:property1} (ii). Thus  $K$ is virtually semisimple.
 \end{proof}

We conclude the paper with the following corollary that   gives a partial solution to Question 1.9
raised in the  introduction. In fact the following is an  answer
to the question in the case that  ``every left and every right
cyclic $R$-module  is virtually semisimple".  However, finding the
structure of non-commutative left fully virtually semisimple rings
(rings each of whose left  cyclic $R$-modules is  virtually
semisimple)  is still an open question.

 \begin{Cor}
The following statements  are equivalent for a ring $R$.\vspace{2mm}\\
{\rm (1)} Every left and every right  cyclic $R$-module  is virtually semisimple (i.e.,   $R$ is a  fully \indent virtually semisimple ring).\vspace{1mm}\\
{\rm (2)}  $R\cong \prod_{i=1}^k M_{n_i}(D_i)$ where each $D_i$ is a principal ideal V-domain.
\end{Cor}

\begin{proof}
$(1) \Rightarrow (2)$.  By Proposition \ref{pro:left fully completely}, the ring  $R$ is a left and a right completely virtually semisimple ring. Thus by Theorem \ref{the:maintheorem},
  $R\cong \prod_{i=1}^k M_{n_i}(D_i)$ where each $D_i$ is a principal ideal domain.  As seen in the proof  $(2)\Rightarrow (3)$ of  Theorem \ref{the:char fully},
we obtain  each  $D_i$ is a left V-domain, and hence by Remark \ref{lem:PLID PRID v-dom},  each $D_i$ is a principal ideal V-domain.

\noindent $(2) \Rightarrow (1)$  is by the second generalization Wedderburn-Artin Theorem.
 \end{proof}

 \end{document}